\documentclass[reqno,11pt]{amsart}
\usepackage{tikz}
\usetikzlibrary{snakes}
\newtheorem{theorem}{Theorem}[section]
\newtheorem{corollary}[theorem]{Corollary}
\newtheorem{proposition}[theorem]{Proposition}

\newtheorem{lemma}[theorem]{Lemma}

\newtheorem{definition}[theorem]{Definition}

\theoremstyle{remark}
\newtheorem{remark}[theorem]{Remark}

\numberwithin{equation}{section}
\allowdisplaybreaks

\newcommand{\N}{\mathbb N}
\newcommand{\naturals}{\mathbb N}
\newcommand{\Z}{\mathbb Z}

\author[Michael J.\ Schlosser]{Michael J.\ Schlosser$^{*}$}
\address{Fakult\"at f\"ur Mathematik, Universit\"at Wien,
Oskar-Morgenstern-Platz~1, A-1090 Vienna, Austria}
\email{michael.schlosser@univie.ac.at}
\urladdr{http://www.mat.univie.ac.at/{\textasciitilde}schlosse}
\thanks{$^{*}$ Partially supported by FWF Austrian Science Fund
grant F50-08 within the SFB
``Algorithmic and enumerative combinatorics''.}

\author[Meesue Yoo]{Meesue Yoo$^{**}$}
\address{AORC center, Sungkyunkwan University, Suwon,16419, Republic of Korea}
\email{meesue.yoo@skku.edu}
\thanks{$^{**}$ Partially supported by the National Research Foundation of
Korea (NRF) grant funded by the Korea government
No.\ 2016R1A5A1008055 and No.\ 2017R1C1B2005653.}

\input xy
\xyoption{all}

\title[Commutation relations and combinatorial
identities]{Weight-dependent commutation relations\\ and
combinatorial identities}



\begin{document}

\begin{abstract}
We derive combinatorial identities for variables satisfying
specific systems of commutation relations, in particular elliptic
commutation relations. The identities thus obtained extend
corresponding ones for $q$-commuting variables $x$ and $y$ satisfying
$yx=qxy$. In particular, we obtain weight-dependent binomial theorems,
functional equations for generalized exponential functions,
we derive results for an elliptic derivative of elliptic commuting variables,
and finally study weight-dependent extensions of the Weyl algebra
which we connect to rook theory.
\end{abstract}

\maketitle
\section{Introduction}
A fundamental question of algebraic combinatorics concerns the study of
connections between algebraic relations and combinatorics. For instance,
the well-studied $q$-commutation relation $yx=qxy$
can be interpreted in terms of weighted lattice paths.
The algebraic expression ``$xy$''
would refer to a path going one step east, then one step north,
while ``$yx$'' would refer to a path with first step north and
second step east. Keeping track of the area ``covered'' by the path
(i.e., the number of square units below the path) and assigning the weight of 
a path $P$ to be $q^{a}$ where $a$ is the area of $P$,
we see that the weight of the path $yx$ is $q$ whereas the weight
of $xy$ is $1$, or, with other words, the path $yx$ has an additional
weight $q$ compared to the path $xy$.
The commutation relation $yx=qxy$ describes exactly the
change of the weights when the two steps are interchanged.

Our purpose here is to carry out a similar analysis
with even more general weights (which depend on the position
of the steps of the path), hereby continuing with a study
that has been commenced by one of us in \cite{Schl1}.
In this work we establish new noncommutative extensions
of the binomial theorem, functional equations for generalized
exponentials, propose an elliptic derivative, and study
weight-dependent extensions of the Weyl algebra which we
connect to rook theory.

Overall, we expect that our findings will not only have applications
to algebraic combinatorics but also to noncommutative analysis,
algebraic geometry, and quantum groups.
This paper is a considerably enhanced and expanded version
(with additional material; in particular on normal orderings)
of an extended abstract published in
the FPSAC'15 proceedings \cite{SY2}.

We would like to thank the reviewers for their careful reading and comments.

\section{Noncommutative weight-dependent binomial theorem}

The following material is taken from the first author's paper \cite{Schl1}.
Let $\mathbb{N}$ and $\mathbb{N}_0$ denote the sets of positive and
nonnegative integers, respectively. 

\begin{definition}
For a doubly-indexed sequence of indeterminates $(w(s,t))_{s,t\in\mathbb{N}}$,
let $\mathbb{C}_w [x,y]$ be the associative unital algebra over $\mathbb{C}$
generated by $x$ and $y$, satisfying the following three relations :
\begin{subequations}\label{eqn:noncommrel}
\begin{align}
yx &= w(1,1)xy,\\
x w(s,t)&= w(s+1,t)x,\\
y w(s,t)&= w(s,t+1)y,
\end{align}
\end{subequations}
for all $s,t\in\mathbb N$. 
\end{definition}

For $s\in\mathbb{N}$ and $t\in \mathbb{N}_0$, we define 
\begin{equation}
W(s,t):= \prod_{j=1}^t w(s,j),
\end{equation}
the empty product being defined to be $1$.
Note that for $s,t\in\mathbb{N}$, we have $w(s,t)=W(s,t)/W(s,t-1)$.
We refer to the $w(s,t)$ as {\em small weights}, whereas to the
$W(s,t)$ as {\em big weights} (or {\em column weights}).

Let the \emph{weight-dependent binomial coefficients} be defined by 
\begin{subequations}\label{wbineq}
\begin{align}\label{recu}
&{}_{\stackrel{\phantom w}{\stackrel{\phantom w}w}}\!\!
\begin{bmatrix}0\\0\end{bmatrix}=1,\qquad
{}_{\stackrel{\phantom w}{\stackrel{\phantom w}w}}\!\!
\begin{bmatrix}n\\k\end{bmatrix}=0
\qquad\text{for\/ $n\in\naturals_0$,\, and\/
$k\in-\naturals$ or $k>n$},\\
\intertext{and}
\label{recw}
&{}_{\stackrel{\phantom w}{\stackrel{\phantom w}w}}\!\!
\begin{bmatrix}n+1\\k\end{bmatrix}=
{}_{\stackrel{\phantom w}{\stackrel{\phantom w}w}}\!\!
\begin{bmatrix}n\\k\end{bmatrix}
+{}_{\stackrel{\phantom w}{\stackrel{\phantom w}w}}\!\!
\begin{bmatrix}n\\k-1\end{bmatrix}
\,W(k,n+1-k)
\qquad \text{for $n,k\in\naturals_0$}.
\end{align}
\end{subequations}

These weight-dependent binomial coefficients have a
combinatorial interpretation in terms of \emph{weighted lattice paths},
see \cite{Schl0}. Here, a lattice path is a sequence of north (or vertical)
and east (or horizontal) steps in the first quadrant of the $xy$-plane,
starting at the origin $(0,0)$ and ending at say $(n,m)$.
We give weights to such paths by assigning the big weight $W(s,t)$
to each east step $(s-1,t)\rightarrow (s,t)$ and $1$ to each north step.
Then define the weight of a path $P$, $w(P)$, to be the product of the
weight of all its steps. 

Given two points $A,B\in\naturals_0^2$,
let $\mathcal{P}(A\rightarrow B)$ be the set of all lattice paths
from $A$ to $B$, and define 
\begin{displaymath}
w(\mathcal{P}(A\rightarrow B)):= \sum_{P\in \mathcal{P}(A\rightarrow B)}w(P).
\end{displaymath}
Then we have 
\begin{equation}
w(\mathcal{P}((0,0)\rightarrow (k,n-k)))=
{}_{\stackrel{\phantom w}{\stackrel{\phantom w}w}}\!\!
\begin{bmatrix}n\\k\end{bmatrix}\label{eqn:wbin}
\end{equation}
as both sides of the equation satisfy the same recursion
and initial condition as in \eqref{wbineq}.

Interpreting the $x$-variable as an east step and the $y$-variable
as a north step, we get the following weight dependent binomial theorem.

\begin{theorem}[\cite{Schl1}]\label{thm:binomial}
Let $n\in\mathbb{N}_0$. Then, as an identity in $\mathbb{C}_w[x,y]$,
\begin{equation}\label{eqn:binomial}
(x+y)^n =\sum_{k=0}^n
{}_{\stackrel{\phantom w}{\stackrel{\phantom w}w}}\!\!
\begin{bmatrix}n\\k\end{bmatrix}x^k y^{n-k}.
\end{equation}
\end{theorem}

Interesting specializations of Theorem~\ref{thm:binomial}
other than the
(classical) binomial theorem and the $q$-binomial theorem include
results involving complete and elementary symmetric functions,
and such involving balanced, well-poised and elliptic weights.
(All of these are featured in \cite{Schl1}.)
Here we shall pay particular attention to the elliptic case;
the corresponding binomial theorem is stated in \eqref{eqn:ab_binom}.
In the following we briefly explain the elliptic setting.


\section{Elliptic weights}

A function is defined to be {\em elliptic} if it is meromorphic and
doubly periodic. It is well known (cf.\ e.g.\ \cite{W}) that
elliptic functions can be built from quotients of
theta functions.

We define the \emph{modified Jacobi theta function} with
\emph{nome} $p$ by
\begin{equation*}
\theta(x;p)=\prod_{j\ge 0}\big((1-p^jx)(1-p^{j+1}/x)\big),\qquad\quad
|p|<1.
\end{equation*}
We write
\begin{equation*}
\theta(x_1,\dots,x_m;p)=\theta(x_1;p)\dots\theta(x_m;p)
\end{equation*}
for products of these functions.
These satisfy the \emph{inversion formula}
\begin{subequations}
\begin{equation}\label{p1id}
\theta(x;p)=-x\theta(1/x; p),
\end{equation}
the \emph{quasi-periodicity property}
\begin{equation}
\theta(px; p)=-\frac{1}{x}\theta(x;p),
\end{equation}
and the \emph{addition formula}
\begin{equation}\label{addf}
\theta(xy,x/y,uv,u/v;p)-\theta(xv,x/v,uy,u/y;p)
=\frac uy\,\theta(yv,y/v,xu,x/u;p)
\end{equation}
\end{subequations}
(cf.\ \cite[p.~451, Example 5]{WhW}).

For nome $p$ with $|p|<1$, \emph{base} $q$, two independent variables
$a$ and $b$, and $(s,t)\in\Z^2$, we define the
\emph{small elliptic weights} to be
\begin{subequations}\label{eqn:ellwt}
\begin{equation}
w_{a,b;q,p}(s,t)=\frac{\theta(aq^{s+2t},bq^{2s+t-2},aq^{t-s-1}/b;p)}
{\theta(aq^{s+2t-2},bq^{2s+t},aq^{t-s+1}/b;p)}q,\label{eqn:ellipticwt}
\end{equation}
and the \emph{big elliptic weights} to be
\begin{equation}
W_{a,b;q,p}(s,t)=\frac{\theta(aq^{s+2t},bq^{2s},bq^{2s-1},
aq^{1-s}/b,aq^{-s}/b;p)}
{\theta(aq^{s},bq^{2s+t},bq^{2s+t-1},aq^{t-s+1}/b,aq^{t-s}/b;p)}q^t.
\label{eqn:ellipticbigwt}
\end{equation}
\end{subequations}
Notice that for $t\ge 0$ we have
\begin{equation*}
W_{a,b;q,p}(s,t)=\prod_{k=1}^t w_{a,b;q,p}(s,k).
\end{equation*}
Observe that
\begin{subequations}
\begin{equation}\label{wshift}
w_{a,b;q,p}(s,k+n)=w_{aq^{2k},bq^k;q,p}(s,n),
\end{equation}
and 
\begin{equation}\label{Wshift}
W_{a,b;q,p}(s,k+n)=W_{a,b;q,p}(s,k)\,
W_{aq^{2k},bq^k;q,p}(s,n),
\end{equation}
\end{subequations}
for all $s$, $k$ and $n$,
which are elementary identities we will make use of.

The terminology ``elliptic'' for the above small and big weights
is indeed justified, as the small weight $w_{a,b;q,p}(s,k)$ (and also the
big weight) is elliptic in each of its parameters (i.e., these weights are
even ``totally elliptic''). Writing $q=e^{2\pi i\sigma}$,
$p=e^{2\pi i\tau}$, $a=q^\alpha$ and $b=q^\beta$ with complex $\sigma$,
$\tau$, $\alpha$, $\beta$, $s$ and $k$, then the small weight
$w_{a,b;q,p}(s,k)$ is clearly periodic in $\alpha$ with period $\sigma^{-1}$.
Also, using \eqref{p1id}, we can see that
$w_{a,b;q,p}(k)$ is also periodic in $\alpha$ with period $\tau\sigma^{-1}$.
The same applies to $w_{a,b;q,p}(k)$ as a function in $\beta$ (or $k$)
with the same two periods $\sigma^{-1}$ and $\tau\sigma^{-1}$.

Next, we define (cf.\  \cite[Ch.\ 11]{GRhyp})
the {\em theta shifted factorial}
(or {\em $q,p$-shifted factorial}), by
\begin{equation*}
(a;q,p)_n = \begin{cases}
\prod^{n-1}_{j=0} \theta (aq^j;p),& n = 1, 2, \ldots\,,\cr
1,& n = 0,\cr
1/\prod^{-n-1}_{j=0} \theta (aq^{n+j};p),& n = -1, -2, \ldots,
\end{cases}
\end{equation*}
and write
\begin{equation*}
(a_1, \ldots, a_m;q, p)_n = (a_1;q,p)_n\ldots(a_m;q,p)_n,
\end{equation*}
for their products.
For $p=0$ we have  $\theta (x;0) = 1-x$ and, hence, $(a;q, 0)_n = (a;q)_n
=(1-a)(1-aq)\dots(1-aq^{n-1})$
is a {\em $q$-shifted factorial} in base $q$.

Now, the \emph{elliptic binomial coefficients} \cite{Schl1}
\begin{equation}\label{ellbin}
\begin{bmatrix}n\\k\end{bmatrix}_{a,b;q,p}:=
\frac{(q^{1+k},aq^{1+k},bq^{1+k},aq^{1-k}/b;q,p)_{n-k}}
{(q,aq,bq^{1+2k},aq/b;q,p)_{n-k}},
\end{equation}
together with the big elliptic weights defined in \eqref{eqn:ellipticbigwt},
can be easily seen to satisfy the recursion \eqref{wbineq}.

In fact, the elliptic binomial coefficients in \eqref{ellbin}
generalize the familiar $q$-binomial coefficients, which can be obtained
by letting $p\to 0$, $a\to 0$, then $b\to 0$, in this order.
These are defined by 
\begin{equation*}
\begin{bmatrix}n\\k\end{bmatrix}_q:=
\frac{(q^{1+k};q)_{n-k}}
{(q;q)_{n-k}},\quad\;\text{where}\quad(a;q)_n= 
\begin{cases}
\prod^{n-1}_{j=0} (1-aq^j),& n = 1, 2, \ldots\,,\cr
1,& n = 0,\cr
1/\prod^{-n-1}_{j=0} (1-aq^{n+j}),& n = -1, -2, \ldots.
\end{cases}
\end{equation*}
(Similarly to above, one writes
$(a_1, \ldots, a_m;q)_n = (a_1;q)_n\ldots(a_m;q)_n,$
for products.)
The $q$-binomial coefficients satisfy the symmetry
\begin{equation*}
\begin{bmatrix}n\\k\end{bmatrix}_q=\begin{bmatrix}n\\n-k\end{bmatrix}_q
\end{equation*}
(which is not satisfied by the elliptic binomial coefficients),
and the two recurrence relations
\begin{align*}
\begin{bmatrix}n+1\\k\end{bmatrix}_q&=
\begin{bmatrix}n\\k\end{bmatrix}_q+
\begin{bmatrix}n\\k-1\end{bmatrix}_q q^{n+1-k},\\
\begin{bmatrix}n+1\\k\end{bmatrix}_q&=
\begin{bmatrix}n\\k\end{bmatrix}_q q^k+
\begin{bmatrix}n\\k-1\end{bmatrix}_q.
\end{align*}
 
Basic hypergeometric series and $q$-series are covered with great detail
in  Gasper and Rahman's textbook \cite{GRhyp};
elliptic hypergeometric series are studied there in Chapter~11.

Let $x,y,a,b$ be four variables with $ab=ba$ and $q,p$ be two complex numbers
with $|p|<1$. We define $\mathbb{C}_{a,b;q,p}[x,y]$ to be the
unital associative algebra over $\mathbb{C}$, generated by $x$ and $y$,
satisfying the following commutation relations 
\begin{subequations}\label{abqxy}
\begin{align}
yx &= \frac{\theta(aq^3,bq,a/bq;p)}{\theta(aq,bq^3,aq/b;p)}qxy,\label{eqn:xy}\\
x\,f(a,b)&= f(aq,bq^2)x,\label{eqn:x}\\
y\,f(a,b)&=  f(aq^2,bq)y,\label{eqn:y}
\end{align}
\end{subequations}
where $f(a,b)$ is any function that is
multiplicatively $p$-periodic in $a$ and $b$
(i.e., which satisfies $f(pa,b)=f(a,pb)=f(a,b)$).

We refer to the variables $x,y,a,b$ forming $\mathbb{C}_{a,b;q,p}[x,y]$
as \emph{elliptic-commuting} variables. The algebra 
$\mathbb{C}_{a,b;q,p}[x,y]$ formally reduces to $\mathbb{C}_q[x,y]$
if one lets $p\to 0$, $a\to 0$, then $b\to 0$ (in this order),
while, having eliminated the nome $p$, relaxing the two conditions
of multiplicative $p$-periodicity.

In $\mathbb{C}_{a,b;q,p}[x,y]$ the following binomial theorem holds
as a consequence of Theorem~\ref{thm:binomial}
(cf.\cite{Schl1}):
\begin{equation}\label{eqn:ab_binom}
(x+y)^n =\sum_{k=0}^n \begin{bmatrix}n\\k\end{bmatrix}_{a,b;q,p}x^k y^{n-k}.
\end{equation}

We can derive the following identity.
\begin{proposition}\label{propbinthm}
For any constant $c$ independent from $a,b$ and $x$,
we have, as an identity in $\mathbb{C}_{a,b;q,p}[x,y]$,
\begin{align}
\prod_{\overleftarrow{k=0}}^{n-1}(1-W_{a,b;q,p}(1,k)cx)
=\sum_{k=0}^n &(-c)^k q^{\binom k2}
\begin{bmatrix}n\\k\end{bmatrix}_{a,b;q,p} \notag
\frac{(aq^n ;q,p)_k}
{(aq^{n-k+1};q,p)_k}\frac{(bq;q,p)_k}{(bq^k;q,p)_k}\\
&\times\frac{(a/b;q^{-1},p)_k(aq/b;q,p)_{n-k}}{(aq^{n-1}/b;q^{-1},p)_k
(aq^{1-k}/b;q,p)_{n-k}}x^k,
\end{align}
where the product 
of noncommuting factors is carried out from right to left
(as the left arrow $\leftarrow$ indicates).
\end{proposition}
\begin{proof}
This is readily proved by induction on $n$ where in the inductive step
one makes use of the addition formula \eqref{addf}.
We leave the details to the reader.
\end{proof}
For $p\to 0$, followed by $a\to 0$ and $b\to 0$ (in this order),
and $x\mapsto x/c$,
Proposition~\ref{propbinthm} reduces to the (commutative version of the)
$q$-binomial theorem:
\begin{equation*}
(x;q)_n=\sum_{k=0}^n (-1)^k q^{\binom k2}
\begin{bmatrix}n\\k\end{bmatrix}_q x^k.
\end{equation*}
 

\subsection{The $p,a\rightarrow 0$ case}

\subsubsection{The $b;q$-binomial theorem}

If one lets $p\to 0$ and then $a\rightarrow 0$ in
$\eqref{eqn:ellwt}$, then the corresponding weights are 
\begin{subequations}
\begin{align}
w_{0,b;q}(s,t)&=\frac{(1-bq^{2s+t-2})}{(1-bq^{2s+t})}q,\\
W_{0,b;q}(s,t)&=\frac{(1-bq^{2s})(1-bq^{2s-1})}
{(1-bq^{2s+t})(1-bq^{2s+t-1})}q^t.
\end{align}
\end{subequations}

Then, in the unital algebra $\mathbb{C}_{0,b;q}[x,y]$ over $\mathbb{C}$
defined by the following three relations 
\begin{subequations}
\begin{align}
yx&= \frac{(1-bq)}{(1-bq^3)}qxy,\label{eqn:b,yx}\\
xb&= q^2 bx,\label{eqn:b,xb}\\
yb&= qby,\label{eqn:b,yb}
\end{align}
\end{subequations}
the following binomial theorem holds : 
\begin{equation}\label{eqn:b-binomial}
(x+y)^n =\sum_{k=0}^n \begin{bmatrix}n\\k\end{bmatrix}_{0,b;q}x^k y^{n-k},
\end{equation}
where 
\begin{equation}\label{eqn:bqbincoeff}
\begin{bmatrix}n\\k\end{bmatrix}_{0,b;q}=
\frac{(q^{1+k},bq^{1+k};q)_{n-k}}{(q,bq^{1+2k};q)_{n-k}}.
\end{equation}
In \eqref{eqn:b-binomial}, by interchanging $k$ and $n-k$
and using the relation 
\begin{displaymath}
x^l y^k =\frac{(bq^{1+k};q)_{2l}}{(bq;q)_{2l}}q^{-kl}y^k x^l,
\end{displaymath}
which can be proved by induction on $l$ and on $k$,
we get 
\begin{align*}
(x+y)^n &= \sum_{k=0}^n \frac{(q^{1+n-k},bq^{1+n-k};q)_k}{(q,bq^{1+2n-2k};q)_k}
x^{n-k}y^k\\
&= \sum_{k=0}^n \frac{(q^{1+n-k},bq^{1+n-k};q)_k}{(q,bq^{1+2n-2k};q)_k}
\times \frac{(bq^{1+k};q)_{2n-2k}}{(bq;q)_{2n-2k}}q^{-k(n-k)}y^k x^{n-k}\\
&=\sum_{k=0}^n \frac{(q;q)_n}{(q;q)_k (q;q)_{n-k}}
\frac{(bq;q)_n}{(bq;q)_k (bq;q)_{n-k}}q^{k(k-n)}y^k x^{n-k}.
\end{align*}
Also, by induction we can derive the following identity 
\begin{equation}
\prod_{\overrightarrow{k=0}}^{n-1}
\left( by +\frac{(bq^2;q^{-1})_k}{(b;q^{-1})_k}x\right)
=\sum_{k=0}^n \begin{bmatrix}n\\k\end{bmatrix}_{0,b;q}
\frac{(bq^{k+2};q)_{n-1}}{(bq^2;q)_{n-1}}q^{k(k-n)}x^k (by)^{n-k},
\end{equation}
where the product (containing the right arrow $\rightarrow$)
is carried out from left to right.

\subsubsection{Identities for $b;q$-exponentials}
Let us define the $b;q$-exponential by 
\begin{equation}\label{bqexp}
e_{b;q}(z):=\sum_{n=0}^{\infty}\frac{1}{(q;q)_n (bq;q)_n}z^n.
\end{equation}
Note that when $b\to 0$, we recover the usual $q$-exponential
(cf.\ \cite[Appendix II, Equation~(II.1)]{GRhyp})
\begin{equation}\label{qpexp}
e_q (z)=e_{0;q}(z)=\sum_{n=0}^{\infty}\frac{z^n}{(q;q)_n}=
\frac{1}{(z;q)_\infty}.
\end{equation}
The following result generalizes the well-known identity for
$q$-exponentials, $e_q(x+y)=e_q(x)e_q(y)$ (a.k.a.\ 
\emph{Cauchy identity})
which was first observed by Sch\"utzenberger~\cite{Sb}. 

\begin{proposition}
In the algebra $\mathbb{C}_{0,b;q}[x,y]$, we have 
\begin{equation}\label{eqn:exp}
e_{b;q}(x+y)=e_{b;q}(x)e_{b;q}(y).
\end{equation}
\end{proposition}

\begin{proof}
We apply the $b;q$-binomial theorem in \eqref{eqn:b-binomial}
to expand $(x+y)^n$ in $e_{b;q}(x+y)$.
Then the left-hand side of \eqref{eqn:exp} can be written as 
\begin{align*}
e_{b;q}(x+y)&=\sum_{n=0}^{\infty}\frac{1}{(q;q)_n (bq;q)_n}(x+y)^n\\
&=\sum_{n=0}^{\infty}\frac{1}{(q;q)_n (bq;q)_n}
\sum_{k=0}^n \begin{bmatrix}n\\k\end{bmatrix}_{0,b;q}x^k y^{n-k}\\
&=\sum_{n=0}^{\infty}\frac{1}{(q;q)_n (bq;q)_n}
\sum_{k=0}^n\frac{(q;q)_n}{(q;q)_k(q;q)_{n-k}}
\frac{(bq;q)_n(bq;q)_{2k}}{(bq;q)_k (bq;q)_{n+k}}x^k y^{n-k}\\
&=\sum_{n=0}^{\infty}\sum_{k=0}^n\frac{1}{(q;q)_k (q;q)_{n-k}
(bq;q)_k(bq^{1+2k};q)_{n-k}}x^k y^{n-k}\\
&=\sum_{n=0}^{\infty}\sum_{k=0}^n\frac{1}{(q,bq;q)_k}x^k
\frac{1}{(q,bq;q)_{n-k}}y^{n-k}\\
&=e_{b;q}(x)e_{b;q}(y).
\end{align*}
Note that we used the relation \eqref{eqn:b,xb}
to shift $x$'s and $b$'s. 
\end{proof}

We shall now look for a product formula for $e_{b;q}(x)$.
\begin{remark}
Let us set 
\begin{equation*}
F_{b;q}(x)=\sum_{n\ge 0}\frac{1}{(q,bq;q)_n}x^n.
\end{equation*}
Then we can verify that $F_{b;q}(x)$ satisfies the following two relations 
\begin{subequations}
\begin{eqnarray}
F_{b;q}(x)-F_{b;q}(qx)&=&\frac{1}{(1-bq)}xF_{b/q;q}(x),\label{eqn:F1}\\
F_{b;q}(x)-bF_{b;q}(qx)&=&(1-b)F_{b/q;q}(x).\label{eqn:F2}
\end{eqnarray}
\end{subequations}
Combining \eqref{eqn:F1} and \eqref{eqn:F2} to
eliminate $F_{b/q;q}(x)$ gives 
\begin{equation*}
\left( 1-\frac{1}{(1-bq)(1-bq^2)}x\right)F_{b;q}(x)=
\left(1-\frac{bq^2}{(1-bq)(1-bq^2)}x \right)F_{b;q}(qx),
\end{equation*}
or 
\begin{equation}\label{eqn:F12}
F_{b;q}(x)=\left( 1-\frac{1}{(1-bq)(1-bq^2)}x\right)^{-1}
\left(1-\frac{bq^2}{(1-bq)(1-bq^2)}x \right)F_{b;q}(qx).
\end{equation}
By iterating \eqref{eqn:F12}, we get 
\begin{equation}
F_{b;q}(x)=\prod_{\overrightarrow{k=0}}^{\infty}
\left[\left( 1-\frac{1}{(1-bq)(1-bq^2)}xq^k\right)^{-1}
\left(1-\frac{bq^2}{(1-bq)(1-bq^2)}xq^k \right) \right],
\end{equation}
where the product 
of noncommuting factors is carried out from left to right
as $k$ increases. This gives a product form for
the $b;q$-exponential $e_{b;q}(x)$. 
\end{remark}

If we put $y\frac{1}{1-b}$ for $z$ in the original $q$-exponential, we get 
\begin{displaymath}
e_q \left( y\frac{1}{1-b}\right) =\sum_{n=0}^{\infty}
\frac{1}{(q;q)_n}\left( y\frac{1}{1-b}\right)^n=
\sum_{n=0}^{\infty}\frac{1}{(q;q)_n (bq;q)_n}y^n =e_{b;q}(y).
\end{displaymath}
Hence, by the product formula for the usual $q$-exponential
in \eqref{qpexp}, we have 
\begin{equation*}
e_{b;q}(y)=\frac{1}{(y(1-b)^{-1};q)_\infty}
\end{equation*}
when the multiplication on the right-hand side in the infinite product
is done from left to right.
Let $u=-xy^{-1}(1-bq)$ and $v=y(1-b)^{-1}$.
Then these two new variables satisfy 
\begin{equation*}
vu=quv.
\end{equation*}
Thus, by the properties of the $q$-exponential (cf.\ \cite{K})
\begin{subequations}
\begin{align}
e_q (u) e_q (v)&=e_q (u+v),\label{eqn:qexp1}\\
e_q (v) e_q (u)&= e_q(u)e_q(-uv)e_q(v),\label{eqn:qexp2}
\end{align}
\end{subequations}
together with $e_q(v)=e_{b;q}(y)$ and $-uv=x$, we get 
\begin{equation}
e_{b;q}(y)e_q (-xy^{-1}(1-bq))=
e_q(-xy^{-1}(1-bq))e_q (x)e_{b;q}(y).\label{eqn:exp3}
\end{equation}
Now
\begin{align*}
e_q (-xy^{-1}(1-bq))&=\sum_{n=0}^{\infty}\frac{1}{(q;q)_n}(-x y^{-1}(1-bq))^n\\
&=\sum_{n=0}^{\infty}\frac{(bq^2 ;q)_n}{(q;q)_n}(-x y^{-1})^n.
\end{align*}
So, if we define 
\begin{equation}
E_{b;q}(z)=\sum_{k=0}^\infty \frac{(bq^2 ;q)_k}{(q;q)_k}(-z)^k,
\end{equation}
then \eqref{eqn:exp3} can be rewritten as 
\begin{equation}
e_{b;q}(y)E_{b;q}(xy^{-1})=E_{b;q}(xy{^{-1}})e_q (x)e_{b;q}(y).
\end{equation}

\subsection{The $p,b\rightarrow 0$ case}

\subsubsection{The $a;q$-binomial theorem}
If one lets $p\to 0$ and then $b\rightarrow 0$ in
\eqref{eqn:ellwt}, then the corresponding
weights are
\begin{subequations}
\begin{align}
w_{a,0;q}(s,t)&=\frac{(1-aq^{s+2t})}{(1-aq^{s+2t-2})}q^{-1},\\
W_{a,0;q}(s,t)&=\frac{(1-aq^{s+2t})}{(1-aq^{s})}q^{-t}.
\end{align}
\end{subequations}

Then, in the unital algebra $\mathbb{C}_{a,0;q}[x,y]$ over $\mathbb{C}$ defined
by the following three relations
\begin{subequations}
\begin{align}
yx&= \frac{(1-aq^{3})}{(1-aq)}q^{-1}xy,\label{eqn:a,yx}\\
xa&= q ax,\label{eqn:a,xa}\\
ya&= q^{2}ay,\label{eqn:a,ya}
\end{align}
\end{subequations}
the following binomial theorem holds : 
\begin{equation}\label{eqn:a-binomial}
(x+y)^n =\sum_{k=0}^n \begin{bmatrix}n\\k\end{bmatrix}_{a,0;q}x^k y^{n-k},
\end{equation}
where 
\begin{equation}\label{eqn:aqbincoeff}
\begin{bmatrix}n\\k\end{bmatrix}_{a,0;q}=
\frac{(q^{1+k},aq^{1+k};q)_{n-k}}{(q,aq;q)_{n-k}}q^{k(k-n)}.
\end{equation}

\subsection{Elliptic case}

\subsubsection{An elliptic derivative operator}
\begin{definition}
Define operators acting on elements in $\mathbb{C}_{a,b;q,p}[x]$ by
\begin{subequations}
\begin{align}
\mathcal{D}_{a,b;q,p}(c(a,b;q,p)x^n)&=
c(aq^{-1},bq^{-2};q,p)
\begin{bmatrix}n\\n-1\end{bmatrix}_{a,b;q,p}x^{n-1}\notag\\
&=c(aq^{-1},bq^{-2};q,p)
\frac{\theta(q^n,aq^n,bq^n,aq^{2-n}/b;p)}
{\theta(q,aq,bq^{2n-1},aq/b;p)}x^{n-1},\\\label{def:eta}
\eta_{a,b;q,p}(c(a,b;q,p)x^n)&=c(a,b;q,p)
\frac{\theta(aq^{1+n},aq^{2+n},bq,aq^{1-n}/b,aq^{-n}/b;p)}
{\theta(aq,aq^2,bq^{1+2n},aq/b,a/b;p)}q^n x^n,
\end{align}
\end{subequations}
where $c(a,b;q,p)$ is any function depending on $a,b,q$ and $p$
but independent from $x$, being multiplicatively $p$-periodic in $a$ and $b$.
\end{definition}

These operators satisfy the following relation (``Pincherle derivative''):
\begin{equation}
\mathcal{D}_{a,b;q,p}x -x\mathcal{D}_{a,b;q,p}=\eta_{a,b;q,p}.
\end{equation}

More generally, the following holds. 
\begin{proposition} For $k\ge 1$, 
\begin{equation}
\mathcal{D}_{a,b;q,p}^{k}x-x\mathcal{D}_{a,b;q,p}^{k} =
\begin{bmatrix}k\\1\end{bmatrix}_{bq^{2-2k},aq^{1-k};q,p}\,
\mathcal{D}_{a,b;q,p}^{k-1}(\eta_{a,b;q,p}).
\end{equation}
\end{proposition}

\begin{proof}
We apply both sides of the operator equation to $c(a,b;q,p)x^n$
and verify that the results are the same.
More precisely, 
\begin{align*}
&\mathcal D_{a,b;q,p}^k x(c(a,b;q,p)x^n)=
\mathcal D_{a,b;q,p}^k(c(aq,bq^2;q,p)x^{n+1})\\
&= c(aq^{1-k},bq^{2-2k};q,p)
\frac{\theta(aq^{n-k+2},aq^{k-n}/b;p)^k}
{\theta(q,bq^{2n-2k+3};p)^k}
\frac{(q^{n+1};q^{-1},p)_k (bq^{n-2k+3};q,p)_k}
{(aq;q^{-1},p)_k (aq/b;q,p)_k}x^{n-k+1},
\end{align*}
and 
\begin{align*}
&x \mathcal D_{a,b;q,p}^k (c(a,b;q,p)x^n)\\&
=c(aq^{1-k},bq^{2-2k};q,p)
\frac{\theta(aq^{n-k+2},aq^{k-n}/b;p)^k}
{\theta(q,bq^{2n-2k+3};p)^k}
\frac{(q^n;q^{-1},p)_k (bq^{n-2k+4};q,p)_k}
{(aq^2;q^{-1},p)_k (a/b;q,p)_k}x^{n-k+1},
\end{align*}
hence
\begin{align}\label{eqn:prop3.5lhs}
&(\mathcal D_{a,b;q,p}^k x -x\mathcal D_{a,b;q,p}^k)
(c(a,b;q,p)x^n)\notag\\
&=c(a q^{1-k},bq^{2-2k};q,p)
\frac{\theta(aq^{n-k+2},aq^{k-n}/b;p)^k}
{\theta(q,bq^{2n-2k+3};p)^k}
\frac{(q^n;q^{-1},p)_{k-1} (bq^{n-2k+4};q,p)_{k-1}}
{(aq^2;q^{-1},p)_{k+1} (a/b;q,p)_{k+1}}\notag\\&\quad\,\times
\left(\theta(q^{n+1},aq^2,bq^{n-2k+3},a/b;p)-
\theta(q^{n-k+1},aq^{2-k},bq^{n-k+3},aq^k/b;p)
\right)x^{n-k+1}\notag\\
&=c(a q^{1-k},bq^{2-2k};q,p)
\frac{\theta(aq^{n-k+2},aq^{k-n}/b;p)^k}
{\theta(q,bq^{2n-2k+3};p)^k}
\frac{(q^n;q^{-1},p)_{k-1} (bq^{n-2k+4};q,p)_{k-1}}
{(aq^2;q^{-1},p)_{k+1} (a/b;q,p)_{k+1}}\notag\\&\quad\,\times
\theta(q^k,aq^{n-k+3},bq^{2-k},aq^{k-n-1}/b;p)
q^{n-k+1}x^{n-k+1},
\end{align}
where in the last step, we have applied the
addition formula for theta functions \eqref{addf}.
On the other hand, 
\begin{equation*}
\begin{bmatrix}k\\1\end{bmatrix}_{bq^{2-2k},aq^{1-k};q,p}=
\frac{\theta(q^k,aq^{3-k},bq^{2-k},aq^{k-1}/b;p)}
{\theta(q,aq^2,bq^{3-2k},a/b;p)}q^{1-k}
\end{equation*}
and 
\begin{align*}
\mathcal D_{a,b;q,p}^{k-1}&\big(\eta_{a,b;q,p}(c(a,b;q,p)x^n)\big)=
c(aq^{1-k},bq^{2-2k};q,p)\\&\times
\frac{\theta(aq^{2+n-k},aq^{3+n-k},bq^{3-2k},aq^{k-n}/b,aq^{k-n-1}/b;p)}
{\theta(aq^{2-k},aq^{3-k},bq^{3+2n-2k},aq^k/b,aq^{k-1}/b;p)}q^n\\
&\times
\frac{\theta(aq^{n-k+2},aq^{k-n}/b;p)^{k-1}}
{\theta(q,bq^{2n-2k+3};p)^{k-1}}
\frac{(q^n;q^{-1},p)_{k-1} (bq^{n-2k+4};q,p)_{k-1}}
{(aq;q^{-1},p)_{k-1} (aq/b;q,p)_{k-1}}x^{n-k+1}.
\end{align*}
It is easy to combine these expressions to confirm that
$$\begin{bmatrix}k\\1\end{bmatrix}_{bq^{2-2k},aq^{1-k};q,p}
\mathcal D_{a,b;q,p}^{k-1}\big(\eta_{a,b;q,p}(c(a,b;q,p)x^n)\big)$$ equals 
\eqref{eqn:prop3.5lhs}.
\end{proof}


\subsubsection{Elliptic Fibonacci numbers}


Recall the sequence of polynomials in $q$ considered by Schur \cite{HW}
defined by
\begin{equation}\label{eqn:qFibonacci}
S_n (q) =S_{n-1}(q)+q^{n-2}S_{n-2}(q),\qquad \text{ for } n>1,
\end{equation}
with $S_n (0)=0$ and $S_n (1)=1$. These can be specialized to the 
sequence of ordinary Fibonacci numbers when $q=1$. 
We provide an elliptic extension of this family of $q$-Fibonacci
polynomials by adding the parameters $a,b,p$ with $|p|<1$
($p$ being the nome).

\begin{definition}
The elliptic Fibonacci numbers $S_{n}(a,b;q,p)$ are defined
by the following recursion 
\begin{align}\label{eqn:ellFrec}
 S_n(a,b;q,p) &=S_{n-1}(aq,bq^2;q,p)\\[.2em]
&\quad+
\frac{\theta(aq^{1+n},aq^{2+n},bq^5,aq^{1-n}/b,aq^{-n}/b;p)}
{\theta(aq^3,aq^4,bq^{1+2n},a/bq,a/bq^2;p)}q^{n-2}
S_{n-2}(a q^2,bq^4;q,p),\notag\\
\text{with }~ S_0 (a,b;q,p)&=0,\quad S_1 (a,b;q,p)=1.\notag
\end{align}

\end{definition}

\begin{proposition}\label{prop:ellF}
Let $xf(a,b;p)=f(aq,bq^2;p)x$ for any function which is
multiplicatively $p$-periodic in $a$ and in $b$. Then
the elliptic Fibonacci numbers $S_{n}(a,b;q,p)$ have the
formal generating function
\begin{equation*}
\sum_{n=0}^\infty S_n (a,b;q,p) x^n =\big(1-x-x^2 \eta_{a,b;q,p}\big)^{-1}x,
\end{equation*}
where $\eta_{a,b;q,p}$ is the linear operator defined in \eqref{def:eta}.
\end{proposition}

\begin{proof}
Let $\mathcal S (x)$ denote the left-hand side of the equation, i.e.,
\begin{equation*}
\mathcal S(x) =\sum_{n=0}^\infty S_n (a,b;q,p) x^n .
\end{equation*}
Then
\begin{align*}
 &(1-x-x^2 \eta_{a,b;q,p})\mathcal S(x) \\
& = x + \sum_{n=2}^\infty S_n (a,b;q,p) x^n
-\sum_{n=0}^\infty S_n (aq,bq^2;q,p) x^{n+1}\\
&\qquad -\sum_{n=0}^\infty S_n (a q^2,bq^4;q,p)
\frac{\theta(aq^{3+n},aq^{4+n},bq^5,aq^{-1-n}/b,aq^{-2-n}/b;p)}
{\theta(aq^3,aq^4,bq^{5+2n},a/bq,a/bq^2;p)}q^n x^{n+2}\\
& = x + \sum_{n=2}^\infty S_n (a,b;q,p) x^n -\sum_{n=2}^\infty
S_{n-1} (aq,bq^2;q,p) x^{n} \\
&\qquad -\sum_{n=2}^\infty S_{n-2} (a q^2,bq^4;q,p)
\frac{\theta(aq^{1+n},aq^{2+n},bq^5,aq^{1-n}/b,aq^{-n}/b;p)}
{\theta(aq^3,aq^4,bq^{1+2n},a/bq,a/bq^2;p)}q^{n-2} x^n\\
& = x + \sum_{n=2}^\infty
\bigg(S_n(a,b;q,p) -S_{n-1}(aq,bq^2;q,p)\\
&\qquad\qquad\qquad -
\frac{\theta(aq^{1+n},aq^{2+n},bq^5,aq^{1-n}/b,aq^{-n}/b;p)}
{\theta(aq^3,aq^4,bq^{1+2n},a/bq,a/bq^2;p)}q^{n-2}
 S_{n-2} (a q^2,bq^4;q,p)\bigg)x^n\\
&= x+\sum_{n=2}^\infty 0\cdot x^n=x,
\end{align*}
by the recurrence relation, which settles the claim.
\end{proof}

Proposition~\ref{prop:ellF} clearly extends the classical
relation for the ordinary Fibonacci numbers
\begin{equation*}
\sum_{n=0}^\infty F_nx^n=\frac x{1-x-x^2}.
\end{equation*}

It is well known that the Fibonacci numbers admit
the following explicit formula:
\begin{equation}\label{fibexp}
F_n=\sum_{j\ge 0}\binom{n-1-j}j.
\end{equation}
Is there a generalization with extra parameters?

While we were not able to extend \eqref{fibexp} to the elliptic setting,
we have succeeded in establishing an $a;q$ version.
For this, we consider the $p\to 0$, followed by $b\to 0$
special case of the elliptic Fibonacci numbers,
which we label as $S_n(a;q)$.
These satisfy the recursion
\begin{equation}\label{eqn:aqFrec}
 S_n(a;q) =S_{n-1}(aq;q) +
\frac{(1-aq^{1+n})(1-aq^{2+n})}
{(1-aq^3)(1-aq^4)}q^{2-n}
S_{n-2}(a q^2;q)
\end{equation}
with $S_0 (a;q)=0$, $S_1 (a;q)=1$. 
We assume that $x$ is a variable that shifts $a$ by $q$,
i.e., we have $xa=qax$.
Let the linear operator $\eta_{a;q}$ be defined by
\begin{equation}
\eta_{a;q}(c(a;q)x^n)=c(a;q)
\frac{(1-aq^{1+n})(1-aq^{2+n})}
{(1-aq)(1-aq^2)}q^{-n} x^n,
\end{equation}
where $c(a;q)$ depends on $a$ and $q$ but is independent from $x$.
Then we have the following auxiliary result:
\begin{lemma}
\begin{equation*}
(x+x^2 \eta_{a;q})^n x = \left(\sum_{j=0}^n q^{-nj}
\begin{bmatrix}n\\j\end{bmatrix}_q \frac{(1-a q^{n+j+2})^j
(1-a q^{n+j+3})^j}{(aq^3 ;q)_j (a q^{n+3};q)_j} x^j\right) x^{n+1}.
\end{equation*}
\end{lemma}
The lemma can be proved by induction. We omit the details. 

We are now ready to give our $a;q$-extension of \eqref{fibexp}:
\begin{corollary}
\begin{equation*}
S_n (a;q)=\sum_{j\ge 0}q^{-(n-j-1)j}\begin{bmatrix} n-j-1\\j\end{bmatrix}_q
\frac{(1-a q^{n+1})^{j}(1-a q^{n+2})^j}{(a q^3;q)_j (a q^{n-j-2};q)_j}.
\end{equation*}
\end{corollary}

\begin{proof}
By the $p\to 0$ followed by $b\to 0$ case of Proposition~\ref{prop:ellF},  
\begin{align*}
\sum_{n=0}^\infty S_n (a;q)x^n &= \big(1-x-x^2 \eta_{a;q}\big)^{-1}x\\
&= \sum_{n=0}^\infty (x+x^2 \eta_{a;q})^n x\\
&= \sum_{n=0}^\infty \left(\sum_{j=0}^n q^{-nj}\begin{bmatrix}n\\j\end{bmatrix}_q
\frac{(1-a q^{n+j+2})^j (1-a q^{n+j+3})^j}{(aq^3 ;q)_j (a q^{n+3};q)_j} x^{n+j+1} 
\right)\\
&= \sum_{n=0}^\infty \left(\sum_{j\ge 0} q^{-(n-j-1)j}
\begin{bmatrix}n-j-1\\j\end{bmatrix}_q \frac{(1-a q^{n+1})^j
(1-a q^{n+2})^j}{(aq^3 ;q)_j (a q^{n-j+2};q)_j} x^{n}  \right).
\end{align*}
The corollary follows by taking coefficients of $x^n$
from both sides of the identity.
\end{proof}


\section{Normal ordering problem}


The Weyl algebra is the algebra generated by $x$ and $y$,
with the commutation relation 
$yx =xy+1$. For an element $\alpha$ in the Weyl algebra, the sum 
\begin{equation*}
\alpha=\sum_{i,j}c_{i,j}x^i y^j
\end{equation*}
is called the \emph{normally ordered form} of $\alpha$ and the
coefficients $c_{i,j}$ are 
called the \emph{normal order coefficients} of $\alpha$.
Considering $\alpha$ as a word, 
$\alpha=\alpha_1 \alpha_2\cdots \alpha_n$, $\alpha_i \in \{ x,y\}$,
$\alpha$ represents a Ferrers diagram drawn
by realizing $x$ as a horizontal step and $y$ as a vertical step.
See Figure~\ref{fig:diagram}
for an example of a Ferrers diagram outlined by a word $\alpha=xyxxyxyy$. 
\begin{figure}[ht]
\begin{equation*}
\begin{picture}(60,60)(0,0)
   \multiput(0,0)(0,10){1}{\line(1,0){14}}
   \multiput(15,0)(0,10){1}{\line(0,1){14}}
    \multiput(15,15)(0,10){1}{\line(1,0){14}}
     \multiput(30,15)(0,10){1}{\line(1,0){14}}
      \multiput(45,15)(0,10){1}{\line(0,1){14}}
       \multiput(45,30)(0,10){1}{\line(1,0){14}}
        \multiput(60,30)(0,10){1}{\line(0,1){14}}
        \multiput(60,45)(0,10){1}{\line(0,1){14}}
 \multiput(15,0)(2,0){23}{\line(1,0){1}} 
  \multiput(60,0)(0,2){15}{\line(0,1){1}}  
  \put(3,1){$x$} 
  \put(16, 5){$y$}
  \put(20, 16){$x$}
  \put(35, 16){$x$}
  \put(46, 20){$y$}
  \put(50, 31){$x$}
  \put(61, 35){$y$}
  \put(61, 50){$y$}
  \end{picture}
\end{equation*}
\caption{A Ferrers diagram $B_\alpha$ outlined by a word
$\alpha=xyxxyxyy$}\label{fig:diagram}
 \end{figure}
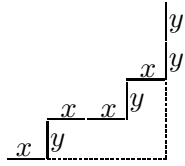
In \cite{Na}, Navon
showed that the normal order coefficients of a word $\alpha$ are
rook numbers on the Ferrers  board outlined by $\alpha$.
(For the definition of a Ferrers board, see Subsection~\ref{subsec:rook}.)
More precisely, let $B_\alpha$ denote the Ferrers diagram outlined 
by $\alpha$, and let $r_k (B_\alpha)$ count the number of ways of choosing
$k$ cells in $B_\alpha$  such that no two cells are in the same column
or in the same row. Then, for a word $\alpha$ composed of $m$ $x$'s and
$n$ $y$'s, the normally ordered form of $\alpha$ is 
\begin{equation}
\alpha=\sum_{k=0}^{\min (m,n)} r_k (B_\alpha) x^{m-k}y^{n-k}.
\end{equation}
For example, in the case of $\alpha=xyxxyxyy$, it is easy to compute 
\begin{align*}
& r_0 (B_\alpha) =1,\\
& r_1 (B_\alpha)=4,\\
& r_2 (B_\alpha) =2,
\end{align*}
and $r_3 (B_\alpha)=0$, $r_4 (B_\alpha)=0$. Hence we have 
$$xyxxyxyy = x^4 y^4 + 4x^3 y^3+2 x^2 y^2.$$
For an excellent survey on the normal ordering of words in the Weyl algebra,
with many references to the literature, see \cite{MS1}.

Here, we extend the normal ordering problem by considering
the following (slightly modified) system of commutation relations,
\begin{subequations}\label{eqn:rookrel}
\begin{align}
yx &= w(1,1)xy+1,\\
x\;\!w(s,t)&= w(s+1,t)x,\\
y\;\!w(s,t)&= w(s,t+1)y,
\end{align}
\end{subequations}
where $(w(s,t))_{s,t\in\mathbb{N}}$ is a doubly-indexed sequence
of indeterminates. This system is just like the one in \eqref{eqn:noncommrel}
but with the first relation altered by having added $1$ to its right-hand side.

We consider the $\mathbb N \times \mathbb N$ grid.
We label the columns from left to right with $1,2,3\dots $ and the
rows from bottom to top  with $1,2,3, \dots$. We use $(i,j)$ to denote
the cell in the $i$-th column from the left and the $j$-th row from
the bottom, and we assign the weight $w(i,j)$ to the $(i,j)$-cell. 
Given a Ferrers board $B$, we say that we place $k$ nonattacking rooks
in $B$ for choosing a $k$-subset of cells in $B$ such that no two cells
lie in the same row or in the same column.
Let $\mathcal N _k (B)$ denote the set of all nonattacking placements
of $k$ rooks in $B$. Given a placement $P\in \mathcal N_k (B)$,
a rook in $P$ \emph{cancels} all the cells to the right in the same row and
all the cells below it in the same column. Then we define
the weighted rook polynomial by 
\begin{equation}
r_k (w;B) =\sum_{P\in \mathcal N_k (B)}
\left(\prod_{(s,t)\in U_B(P)}w(s-r_{P,(s,t)},t)
\right), 
\end{equation}
where $U_B(P)$ is the set of cells in $B-P$ which are uncancelled by
any rooks in $P$ and $r_{P,(i,j)}$ is the number of rooks in the
north-west region of the cell in $(i,j)$. 

\begin{theorem}\label{thm:wnormalorder}
Let $\alpha$ be an element composed of $m$ x's and $n$ y's where
$x$ and $y$ are subject to the relations in \eqref{eqn:rookrel}. 
Then the normally ordered form of $\alpha$ is  
\begin{equation}\label{eqn:normalorder}
\alpha =\sum_{k=0}^{\min (m,n)} r_k (w;B_\alpha)x^{m-k} y^{n-k}.
\end{equation}
\end{theorem}

\begin{proof}
First of all, it is not difficult to see that the normal ordered form of any
word $\alpha$ with respect to the system \eqref{eqn:rookrel} is unique
(and does not depend on the order the commutation relations are performed). 
Thus we may carry out the proof by induction on the number of cells
in $B_\alpha$.
The relation $yx=w(1,1)xy+1$ describes the situation when the board has only
one box in $(1,1)$. If we do not place any rook, then the rook
polynomial $r_0 (w;B_{yx})=w(1,1)$, and if we place a rook in the $(1,1)$ cell, 
then $r_1 (w;B_{yx})=1$. More generally, let $\alpha =x^{l-1}yx$.
(Note that $y$'s coming after $yx$ do not change the shape of
the Ferrers board outlined by $\alpha$, so we may omit them without loss of generality.) Then $x^{l-1}yx = w(l,1) x^l y + x^{l-1}$.
The Ferrers board outlined by $\alpha$ has only one cell with the
weight $w(l,1)$ and the coefficients match 
to the rook polynomials of the cases $k=0$ or $k=1$, respectively.

Now consider $\alpha$ such that the Ferrers board $B_\alpha$ has more
than two cells. We assume the theorem holds for words corresponding
to Ferrers boards with a smaller number of cells than $B_\alpha$. 
Let us find the right-most $yx$ in $\alpha$ and write $\alpha$ 
as $\alpha = \alpha ' yx \alpha''$, where $\alpha''$ would be of
the form $x^s y^t$. 
If we apply the commutation relation to $yx$, then we get 
\begin{align*}
\alpha &= \alpha' yx \alpha''\\
&=\alpha' (w(1,1)xy +1)\alpha'' \\
&=w(s', t')\alpha' xy \alpha'' +\alpha' \alpha'',
\end{align*}
where $s'-1$ is the number of $x$'s in $\alpha'$ and $t'-1$ is the
number of $y$'s in $\alpha'$. The Ferrers diagram outlined by
$\alpha' xy \alpha''$ is the diagram obtained by deleting the 
cell $(s',t')$ from $B_\alpha$, and $\alpha' \alpha''$ corresponds
to the Ferrers diagram of $B_\alpha$ after removing the column and
row containing the cell $(s', t')$. 
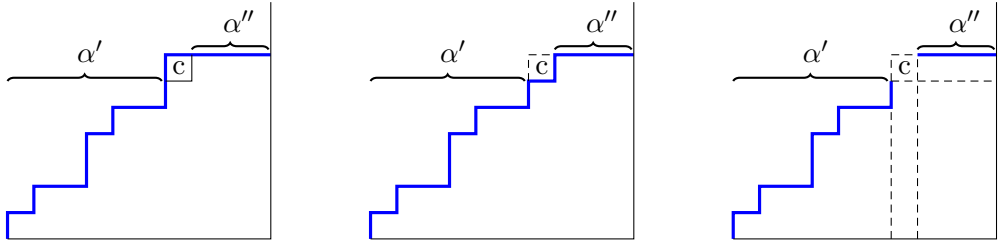
\begin{figure}[ht]
\begin{tikzpicture}
\draw[blue, very thick] (0,0) -- (0,.35) --(.35, .35)--(.35,.7)--(1.05,.7)--(1.05, 1.4)--(1.4, 1.4)--(1.4, 1.75)--(2.1,1.75)--(2.1,2.45)--(3.5,2.45);
\draw (0,0)--(3.5,0)--(3.5,3.15);
\draw (2.1,2.1)--(2.45, 2.1) --(2.45,2.45);
\draw[snake=brace, thick]   (0, 2.1) -- (2.05, 2.1);
\draw[snake=brace, thick] (2.45, 2.52) -- (3.47, 2.52);
\node (c) at (2.28, 2.28) {c};
\node (a) at (1.1, 2.5) {$\alpha'$};
\node (b) at (3.08, 2.9) {$\alpha''$};
\end{tikzpicture}
\qquad\quad
\begin{tikzpicture}
\draw[blue, very thick] (0,0) -- (0,.35) --(.35, .35)--(.35,.7)--(1.05,.7)--(1.05, 1.4)--(1.4, 1.4)--(1.4, 1.75)--(2.1,1.75)--(2.1,2.1)--(2.45,2.1)--(2.45,2.45)--(3.5,2.45);
\draw (0,0)--(3.5,0)--(3.5,3.15);
\draw[densely dashed] (2.1,2.1)--(2.1, 2.45) --(2.45,2.45);
\draw[snake=brace, thick]   (0, 2.1) -- (2.05, 2.1);
\draw[snake=brace, thick] (2.45, 2.52) -- (3.47, 2.52);
\node (c) at (2.28, 2.28) {c};
\node (a) at (1.1, 2.5) {$\alpha'$};
\node (b) at (3.08, 2.9) {$\alpha''$};
\end{tikzpicture}
\qquad\quad
\begin{tikzpicture}
\draw[blue, very thick] (0,0) -- (0,.35) --(.35, .35)--(.35,.7)--(1.05,.7)--(1.05, 1.4)--(1.4, 1.4)--(1.4, 1.75)--(2.1,1.75)--(2.1,2.1);
\draw[blue, very thick] (2.45,2.45)--(3.5,2.45);
\draw[densely dashed] (2.1,0)--(2.1, 1.75);
\draw[densely dashed] (2.45,0)--(2.45, 2.45);
\draw[densely dashed] (2.1,2.1)--(3.5, 2.1);
\draw (0,0)--(3.5,0)--(3.5,3.15);
\draw[densely dashed] (2.1,2.1)--(2.1, 2.45) --(2.45,2.45);
\draw[snake=brace, thick]   (0, 2.1) -- (2.05, 2.1);
\draw[snake=brace, thick] (2.45, 2.52) -- (3.47, 2.52);
\node (c) at (2.28, 2.28) {c};
\node (a) at (1.1, 2.5) {$\alpha'$};
\node (b) at (3.08, 2.9) {$\alpha''$};
\end{tikzpicture}
\caption{$B_\alpha$, $B_{\alpha' xy \alpha''}$ and $B_{\alpha'\alpha''}$,
from the left.}\label{fig:F}
\end{figure}

Figure~\ref{fig:F} shows diagrams corresponding to $\alpha$,
$\alpha xy\alpha''$ and $\alpha'\alpha''$, from the left.
In Figure~\ref{fig:F}, $c$ denotes the $(s', t')$ cell. 
As we can see in Figure~\ref{fig:F}, $B_{\alpha' xy\alpha''}$ and
$B_{\alpha'\alpha''}$ have a smaller number of cells than $B_\alpha$ and so we
can use the induction hypothesis. Thus
\begin{align*}
\alpha &= w(s', t')\alpha' xy\alpha'' +\alpha' \alpha''\\
&= w(s', t')\left(\sum_{k=0}^{\min(m,n)}r_k (w;B_{\alpha' xy\alpha''})x^{m-k}y^{n-k}
\right)\\
&\qquad  +\sum_{k=0}^{\min (m,n)-1}r_k (w;B_{\alpha' \alpha''})x^{m-1-k}y^{n-1-k}\\
&=\underbrace{ w(s', t')r_0
(w;B_{\alpha' xy \alpha''})}_{=\prod_{(i,j)\in B_\alpha}w(i,j)}x^m y^n\\
&\qquad +\sum_{k=1}^{\min (m,n)}\left(w(s', t')r_k (w; B_{\alpha' xy \alpha''})
+r_{k-1}(w;B_{\alpha' \alpha''}) \right)x^{m-k}y^{n-k}\\
&= \sum_{k=0}^{\min (m,n)}r_k (w;B_{\alpha}) x^{m-k}y^{n-k},
\end{align*}
since $r_k (w;B_{\alpha})=w(s', t')r_k (w; B_{\alpha' xy \alpha''})+
r_{k-1}(w;B_{\alpha' \alpha''})$
constitutes a recursion by distinguishing between the
cases when there is a rook in the $(s', t')$ cell or not.
\end{proof}


\subsection{Extension of rook theory}\label{subsec:rook}


A \emph{board} is a finite subset of $\mathbb N \times \mathbb N$.
Let $b_1,\dots,b_n\in\mathbb N_0$.
We use the notation $B(b_1,\dots, b_n)$ to denote the set of cells 
\begin{equation*}
B(b_1,\dots, b_n)=\{(i,j) ~|~ 1\le i \le n,~ 1\le j\le b_i\}.
\end{equation*}
In the special case when the 
$b_i$'s are nondecreasing, i.e., $b_1\le b_2\le \cdots \le b_n$,
the board is called a \emph{Ferrers board}.
Recall that $\mathcal{N}_{k}(B)$ denotes the set of all $k$-rook placements
in $B$ such that no two rooks lie in the same row or column.
Garsia and Remmel~\cite{GR} introduced a $q$-analogue of the rook numbers
for Ferrers boards,
\begin{displaymath}
r_k (q;B)=\sum_{P\in \mathcal{N}_k (B)}q^{|U_B(P)|},
\end{displaymath}
where $|U_B(P)|$ counts the number of uncancelled cells in $B-P$,
and they proved the
product formula
\begin{equation}\label{qrookthm}
\prod_{i=1}^n [z+b_i-i+1]_q =
\sum_{k=0}^n r_{k}(q;B)[z]_q\!\downarrow_{n-k},
\end{equation} 
where $B=B(b_1,\dots, b_n)$, $[n]_q =\frac{1-q^n}{1-q}$ and
$[n]_q\!\downarrow_k=[n]_q [n-1]_q \dots [n-k+1]_q$.
The formula in \eqref{qrookthm} is a $q$-analogue 
of the product formula shown by Goldman, Joichi and White \cite{GJW}.

In \eqref{eqn:rookrel}, in the case when $w(s,t)=q$ for all $s$ and $t$,
the variables $x$ and $y$ generate the $q$-Weyl algebra.
Theorem~\ref{thm:wnormalorder}
can be used to prove the product formula \eqref{qrookthm} by choosing
$y=D_q$ where
\begin{equation*}
D_q (f(t))=\frac{f(qt)-f(t)}{(q-1)t}
\end{equation*} 
and $x$ is the operator acting by multiplication by $t$.
The product formula \eqref{qrookthm} is then obtained by applying the
left- and the right-hand sides of \eqref{eqn:normalorder}
to $t^z$ \cite{Var}. In \cite{SY}, we have established an elliptic
analogue of rook numbers for Ferrers boards
by assigning elliptic weights to cells of the Ferrers board
and proved an elliptic analogue of the product formula.
Here we utilize Theorem~\ref{thm:wnormalorder}
to prove the elliptic analogue of the product formula in a different way.
To make the notation simpler, we let 
\begin{align*}
w_{a,b;q,p}(k)&:= w_{a,b;q,p}(1,k)=\frac{\theta(aq^{2k+1},bq^{k},aq^{k-2}/b;p)}
{\theta(aq^{2k-1},bq^{k+2},aq^k /b;p)}q,\\
W_{a,b;q,p}(k) &:= W_{a,b;q,p}(1,k)=
\frac{\theta(a q^{1+2k},bq, bq^2, a q^{-1}/b, a/b;p)}
{\theta(aq, bq^{k+1},bq^{k+2},aq^k/b;p)}q^k,\\ 
[z]_{a,b;q,p}&:=\begin{bmatrix}z\\
1\end{bmatrix}_{a,b;q,p}=\frac{\theta(q^z, aq^z, bq^2, a/b;p)}
{\theta(q, aq, bq^{z+1},aq^{z-1}/b;p)}.
\end{align*}
Note that 
\begin{equation}
w_{a,b;q,p}(j+k)=w_{aq^{2j},bq^{j};q,p}(k),
\end{equation}
which we later make use of.

\begin{definition}
Given a Ferrers board $B$, we define the elliptic analogue of the
$k$-rook number by 
\begin{displaymath}\label{def:elptrook}
r_k(a,b;q,p;B)=\sum_{P\in \mathcal{N}_k(B)}\text{wt}(P),
\end{displaymath}
where 
\begin{displaymath}\text{wt}(P)=\prod_{(i,j)\in U_B(P)}
w_{a,b;q,p}(i-j-r_{P,(i,j)}),
\end{displaymath}
and $r_{P,(i,j)}$
is the number of rooks in $P$ which are in the north-west region of $(i,j)$. 
\end{definition}
Then in \cite[Theorem~12]{SY}, it was proved that $r_k(a,b;q,p;B)$ satisfies
an analogous identity to the product formula in \eqref{qrookthm}.
The proof proceeded along the lines of Garsia and Remmel's \cite{GR} proof
of \eqref{qrookthm} suitably adapted to the elliptic setting.
Here we use Theorem~\ref{thm:wnormalorder} to provide a different proof
of the elliptic product formula.
For this theorem, we assume that a Ferrers board $B=B(b_1,\dots, b_n)$
with $n$ columns 
is contained in $[n]\times [n]$ grid i.e., $b_n\le n$.
This can be achieved if needed by starting with sufficiently many
columns of height zero.

\begin{theorem}\label{thm:aqrook}
For a Ferrers board $B=B(b_1,\dots, b_n$), with
$b_1\le b_2 \le \cdots \le b_n$, we have 
\begin{align}\label{eqn:aqrook}
&\prod_{i=1}^n [z+b_i -i+1]_{aq^{2(i-1-b_i)},bq^{i-1-b_i};q,p}\notag\\
&=\sum_{k=0}^{n}r_{n-k}(a,b;q,p;B)\prod_{j=1}^k
[z-j+1]_{aq^{2(j-1)},b q^{j-1};q,p}.
\end{align}
\end{theorem}

\begin{proof}
For the specific elliptic weight used to define 
the elliptic analogue of the rook number in Definition~\ref{def:elptrook},
assume the noncommuting variables $x$ and $y$ to satisfy the system
of relations in \eqref{eqn:rookrel} where $w(s,t):=w_{a,b;q,p}(s-t)$, 
specifically, 
\begin{subequations}\label{eqn:rookrelweyl}
\begin{align}
yx& = w_{a,b;q,p}(0)xy+1,\\
x\;\! w_{a,b;q,p}(k) &= w_{a,b;q,p}(k+1)x,\\
y\;\! w_{a,b;q,p}(k)& = w_{a,b;q,p}(k-1)y.
\end{align}
\end{subequations}
Now, define a differential operator acting on polynomials in $t$ by 
\begin{equation*}
\mathcal D_{a,b;q,p}(f(a,b;q,p)\,t^z)=
f(a q^{-2}, b q^{-1};q,p)[z]_{a q^{-2}, bq^{-1};q,p}\,t^{z-1}
\end{equation*}
and let $y$ be the linear operator $\mathcal D_{a,b;q,p}$.
Let $x$ be an operator 
acting by multiplying by $t$ and shifting $a$ by $q^2$ and $b$ by $q$, i.e., 
\begin{equation*}
x(f(a,b;q,p)\,t^z) =f(a q^2, bq;q,p)\,t^{z+1}.
\end{equation*}
Then, using
\begin{equation*}
[z+1]_{aq^{-2},bq^{-1};q,p}-w_{a,b;q,p}(0)[z]_{a,b;q,p}=1
\end{equation*}
(which follows from the addition formula \eqref{addf}),
we can check that the $x$ and $y$ operators satisfy the three
relations in \eqref{eqn:rookrelweyl}.

Any Ferrers diagram $B\subseteq [n]\times [n]$ corresponds to a
word $\alpha_B$ in the letters $x$ and $y$. In fact, the
Ferrers boards is outlined by a path
starting from the bottom-left corner up to the 
top-right corner of $[n]\times [n]$ grid,
consisting of horizontal and vertical steps,
which are assigned the letters $x$ and $y$, respectively.
For instance, the Ferrers diagram 
$B(1,2,2,3)$ in Figure \eqref{fig:wd} corresponds to the word
$\alpha_B =yxyxxyxy$.
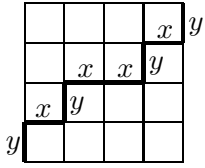
\begin{figure}[ht]
\begin{equation*}
\begin{picture}(70,60)(0,0)
\multiput(0,0)(0,15){5}{\line(1,0){60}}
\multiput(0,0)(15,0){5}{\line(0,1){60}}
\thicklines \linethickness{1.3pt}
\multiput(0,0)(0,15){1}{\line(0,1){15}}
\multiput(0,15)(0,15){1}{\line(1,0){15}}
\multiput(15,15)(0,15){1}{\line(0,1){15}}
\multiput(15,30)(0,15){1}{\line(1,0){30}}
\multiput(45,30)(0,15){1}{\line(0,1){15}}
\multiput(45,45)(0,15){1}{\line(1,0){15}}
\multiput(60,45)(0,15){1}{\line(0,1){15}}
\put(-7,5){$y$}
\put(17, 20){$y$}
\put(47, 35){$y$}
\put(62, 50){$y$}
\put(4,17){$x$}
\put(20,32){$x$}
\put(35,32){$x$}
\put(50,47){$x$}
\end{picture}
\end{equation*}
\caption{$B(1,2,2,3)$ }\label{fig:wd}
\end{figure}
Note that the assumption $B\subseteq [n]\times [n]$ implies that
the corresponding word $\alpha_B$ has $n$ $x$'s and $n$ $y$'s.
By Theorem~\ref{thm:wnormalorder}
we have 
\begin{equation}\label{eqn:alphaB}
\alpha_B =\sum_{k=0}^n r_{n-k} (a,b;q,p;B) x^{k} y^{k}
\end{equation}
We apply each side of \eqref{eqn:alphaB} to $t^z$. 

If we apply $x^k y^k$ to $t^z$, then we get
\begin{equation*}
x^k y^k(t^z)=[z]_{a,b;q,p}[z-1]_{aq^2, bq;q,p} \cdots
[z-k+1]_{aq^{2(k-1),bq^{k-1};q,p}}t^z
\end{equation*}
after applying successively all the $y$'s and the $x$'s to the left of $t^k$.
Hence, applying the right-hand side of \eqref{eqn:alphaB} to $t^z$ gives
\begin{equation*}
\sum_{k=0}^n r_{n-k} (a,b;q,p;B) [z]_{a,b;q,p}[z-1]_{aq^2, bq;q,p} \cdots
[z-k+1]_{aq^{2(k-1),bq^{k-1};q,p}}t^z.
\end{equation*}

Let us consider applying the left-hand side of \eqref{eqn:alphaB} to $t^z$.
Let $\tilde{B}$ denote the conjugate Ferrers diagram of $B$, obtained
from the diagram $B$ by reflecting it about the the anti-diagonal
in the $[n]\times [n]$ grid.
If $\tilde{B}=B(c_1,\dots, c_n)$, then $c_i$ is the
number of cells in the $i$th row of $B$, reading from the top.
Now we apply $\alpha$ to $t^z$. Note that applying each $y$ produces 
a factor and applying $x$ increases the exponent of $t$, while the 
number of $x$'s to the right of the $j$-th $y$ from the right
is equal to $c_j$. 
Then applying the $j$-th $y$ from the right to $t^z$ produces a factor 
$[z+c_j-j+1]_{aq^{-2},bq^{-1};q,p}$, and since there are $(n-j)$ $y$'s and 
$(n-c_j)$ $x$'s to the left of $[z+c_j-j+1]_{aq^{-2},bq^{-1};q,p}$, 
the factor becomes $[z+c_j -j+1]_{aq^{-2(c_j -j+1)},bq^{-(c_j -j+1)};q,p}$.
Hence, applying the left-hand side of \eqref{eqn:alphaB} to $t^z$ gives 
\begin{equation*}
\prod_{i=1}^n [z+c_j -j+1]_{aq^{-2(c_j -j+1)},bq^{-(c_j -j+1)};q,p }t^z.
\end{equation*}
Since
$$\prod_{i=1}^n [z+c_j -j+1]_{aq^{-2(c_j -j+1)},bq^{-(c_j -j+1)};q,p }=
\prod_{i=1}^n [z+b_j -j+1]_{aq^{-2(b_j -j+1)},bq^{-(b_j -j+1)};q,p }$$
for $B(b_1,\dots, b_n)$ and $B(c_1,\dots, c_n)$ being conjugates of
each other, by taking coefficients of $t^z$ from both sides,
we have proved the theorem.
\end{proof}




In \cite{SY}, we have also considered file numbers and established an
elliptic analogue for these. 

Given a board $B\subseteq [n]\times [n]$, a \emph{file placement}
$Q$ of $k$ rooks in $B$ is a $k$ subset of $B$ such that no two cells
in $Q$ lie in the same column.
Thus, we allow the placement of multiple rooks in the same row. 
Let $\mathcal F_k (B)$ be the set of all file placements of $k$ rooks in $B$. 
Given a placement $Q\in \mathcal F_k (B)$, a rook in $Q$ cancels all the cells
below it in the same column. Due to this cancellation scheme, we only need to 
assume that the given board is a skyline board. However, since we are
considering
boards outlined by words composed of $x$'s and $y$'s satisfying certain 
commutation relations, we assume that a given board is a Ferrers board.

Given a Ferrers board $B$, define the weighted file polynomial by 
\begin{equation}
f_k (w;B)=\sum_{Q\in \mathcal F_k (B)} \left( \prod_{(s,t)\in U_B(Q)}
w(s-r_{P, (s,t)},t) \right), 
\end{equation}
where $U_B (Q)$ is the set of uncancelled cells in $B-Q$ and $r_{P,(i,j)}$ 
is the number of cells in the north-west region of the cell $(i,j)$. 

Now, we consider a word $\alpha=\alpha_1 \alpha_2\cdots \alpha_n$,
$\alpha_i \in \{x,y\}$,
where $x$ and $y$ satisfy the following commutation relations
\begin{subequations}\label{eqn:wfilerel}
\begin{align}
yx &=w(1,1)xy +y,\\
x\;\! w(s,t)&= w(s+1, t)x,\\
y\;\!w(s,t)&= w(s, t+1)y,
\end{align}
\end{subequations}
where $(w(s,t))_{(s,t)\in \mathbb N^2}$ a doubly-indexed sequence of
indeterminates.
As we did before, we consider a Ferrers diagram $B_\alpha$ outlined by $\alpha$.
\begin{theorem}\label{thm:filenormalorder}
Let $\alpha$ be a word composed of $m$ $x$'s and $n$ $y$'s where
$x$ and $y$ are subject to 
the relations in \eqref{eqn:wfilerel}. Then
\begin{equation}\label{eqn:filenormalorder}
 \alpha = \sum_{k=0}^m f_k (w;B_\alpha) x^{m-k}y^n.
\end{equation}
\end{theorem}

\begin{proof}
The proof is similar to the proof of Theorem~\ref{thm:wnormalorder}. 
As we did in the proof of Theorem~\ref{thm:wnormalorder}, 
we may carry out the proof by induction on the number
of cells in $B_\alpha$.
The commutation relation $yx=w(1,1)xy+y$ describes the situation when the
Ferrers diagram has one box in the $(1,1)$ cell. The coefficient
$w(1,1)$ in the first term $w(1,1)xy$ corresponds to $f_0 (w;B_{yx})$
when there is no rook in $B_{yx}$, and the term $y$ comes from
the case when a rook is placed in the $(1,1)$ cell.
The fact that $x$ drops out but $y$ remains can be explained as follows:
in this case the column containing a rook is cancelled by the rook,
but the row is still available (for an additional rook)
even after a rook has been placed there.
Figure~\ref{fig:rel1} depicts this situation. In the right-most
diagram, ``$R$'' means that a rook is placed in the cell. 
\begin{figure}[ht]
\begin{equation*}
\begin{picture}(20,20)(0,0)
 \multiput(5,0)(0,15){1}{\line(1,0){15}}
    \multiput(20,0)(0,15){1}{\line(0,1){15}}
    \put(-2,6){$y$}
    \put(9, 17){$x$}
\thicklines \linethickness{1.28pt}
   \multiput(5,15)(0,15){1}{\line(1,0){15}}
    \multiput(5,0)(0,15){1}{\line(0,1){15}}
  \end{picture}
  \qquad =
  \qquad 
  \begin{picture}(20,20)(0,0)
    \multiput(0,15)(2,0){8}{\line(1,0){1}}
     \multiput(0,0)(0,2){8}{\line(0,1){1}}
   \thicklines \linethickness{1.28pt}
   \multiput(0,0)(0,15){1}{\line(1,0){15}}
    \multiput(15,0)(0,15){1}{\line(0,1){15}}
      \put(17,6){$y$}
    \put(5, 2){$x$}
  \end{picture}
  \qquad +
  \qquad
  \begin{picture}(20,20)(0,0)
    \multiput(0,15)(2,0){8}{\line(1,0){1}}
     \multiput(0,0)(0,2){8}{\line(0,1){1}}
     \multiput(0,0)(2,0){8}{\line(1,0){1}}
   \thicklines \linethickness{1.28pt}
    \multiput(15,0)(0,15){1}{\line(0,1){15}}
      \put(17,6){$y$}
      \put(3,4){$R$}
      \end{picture}
\end{equation*}
\caption{Figure corresponding to the relation $yx=w(1,1)xy+y$}\label{fig:rel1}
\end{figure}

Now assume that \eqref{eqn:filenormalorder} holds for diagrams with
a smaller number of cells than the diagram $B_\alpha$ outlined by $\alpha$. 
We look for the right-most consecutive pair $yx$ in $\alpha$ and denote 
the part left of $yx$ by $\alpha'$ and the part right of $yx$ by $\alpha''$. 
Then we have 
\begin{align*}
\alpha &= \alpha' yx \alpha'' \\
&= \alpha' (w(1,1)xy+y)\alpha'' \\
&= w(s', t')\alpha' xy\alpha'' + \alpha' y \alpha'',
\end{align*}
where $s'-1$ is the number of $x$'s in $\alpha'$ and $t'-1$ is the
number of $y$'s in $\alpha'$.  
Since the diagrams outlined by $\alpha' xy\alpha''$ and $\alpha' y\alpha''$ 
have less number of cells than $B_\alpha$, we can use the induction hypothesis. 
Thus,
\begin{align*}
\alpha &= w(s', t')\sum_{k=0}^m f_k(w;B_{\alpha' xy\alpha''})x^{m-k}y^n
+\sum_{k=0}^{m-1}f_k (w;B_{\alpha' y\alpha''})x^{m-1-k}y^n\\
&=\underbrace{w(s', t')f_0(w;B_{\alpha' xy\alpha''})}_{=\prod_{(i,j)\in B\alpha}w(i,j)}
x^my^n\\&\qquad
 +\sum_{k=1}^m (w(s', t')f_k(w;B_{\alpha' xy\alpha''}) +
f_{k-1} (w;B_{\alpha' y\alpha''}))x^{m-k}y^n\\
&= \sum_{k=0}^m f_k (w;B_\alpha) x^{m-k}y^n,
\end{align*}
since $w(s', t')f_k(w;B_{\alpha' xy\alpha''}) + f_{k-1} (w;B_{\alpha' y\alpha''})$
constitutes a recursion for $f_k (w;B_\alpha)$ by distinguishing
between the cases when 
there is a rook or not in the $(s', t')$-cell. 
\end{proof}

In \cite{SY}, we have defined an elliptic analogue of the $k$-th file
numbers by 
\begin{equation}\label{eqn:wfiledef}
f_k (a,b;q,p;B) =\sum_{Q\in \mathcal F_k (B)}
\left(\prod_{(i,j)\in U_B (Q)}w_{a,b;q,p}(1-j) \right)
\end{equation}
for $B$ being a skyline board,
and proved a product formula involving $f_k (a,b;q,p;B)$ by
considering rook placements in an extended board.
Here, assuming the special case that $B$ is a Ferrers board,
(which corresponds to a word $\alpha_B$ outlining $B$)
we prove the product formula, for which we utilize
Theorem~\ref{thm:filenormalorder}.

\begin{theorem}
For a Ferrers board $B=B(b_1, b_2, \dots, b_n)\subseteq [n]\times [n]$, we have 
\begin{equation}
 \prod_{i=1}^n [z+b_i]_{aq^{-2b_i},bq^{-b_i};q,p}=\sum_{k=0}^n f_{n-k}(a,b;q,p;B)
([z]_{a,b;q,p})^k.
\end{equation}
\end{theorem}

\begin{proof}
Starting from the bottom-left corner or the Ferrers diagram $B$,
read out a word of $x$ and $y$, corresponding to the boundary of $B$
where a horizontal step is $x$ and a vertical step is $y$.
Since we assumed that $B\subseteq [n]\times [n]$, the word $\alpha_B$
outlining the board $B$ has $n$ $x$'s and $n$ $y$'s. If we specify
the weight $w(s,t)$ assigned to the $(s,t)$ cell by
$w(s,t)=w_{a,b;q,p}(1-t):= w_{a,b;q,p}(1,1-t)$, then the commutation relations
in \eqref{eqn:wfilerel} become 
\begin{subequations}\label{eqn:elptfilerel}
\begin{align}
yx &=w_{a,b;q,p}(0)xy +y,\\
x\;\! w_{a,b;q,p}(k)& = w_{a,b;q,p}(k)x,\\
y\;\! w_{a,b;q,p}(k)& = w_{a,b;q,p}(k-1)y.
\end{align}
\end{subequations}
Define $x$ to be an operator acting as 
\begin{equation*}
x(f(a,b;q,p)\,t^z)=f(a,b;q,p)[z]_{a,b;q,p}\,t^z,
\end{equation*}
which could be considered as a composition of the differential operator 
\begin{equation*}
\mathcal D_{a,b;q,p}(f(a,b;q,p)\,t^z)=
f(aq^{-2},bq^{-1};q,p)[z]_{aq^{-2},bq^{-1};q,p}\,t^{z-1}
\end{equation*}
and the multiplication operator by $t$ with the shifts 
\begin{equation*}
f(a,b;q,p)\,t^z \mapsto f(aq^2, bq;q,p)\,t^{z+1},
\end{equation*}
and $y$ to be an operator acting as 
\begin{equation*}
y(f(a,b;q,p)t^z)=f(a q^{-2}, bq^{-1};q,p)t^{z-1}.
\end{equation*}
Using
\begin{equation*}
[z]_{aq^{-2},bq^{-1};q,p}-w_{a,b;q,p}(0)[z-1]_{a,b;q,p}=1
\end{equation*}
we can check that $x$ and $y$ satisfy the relations in
\eqref{eqn:elptfilerel}. Hence by Theorem~\ref{thm:filenormalorder},
we have 
\begin{equation}\label{eqn:fileprod}
\alpha_B = \sum_{k=0}^n f_{n-k} (a,b;q,p;B) x^{k}y^n.
\end{equation}

We now apply each side of \eqref{eqn:fileprod} to $t^{z+n}$. 
First of all, since
\begin{equation*}
x^k y^n (t^{z+n})=x^k(t^z) =([z]_{a,b;q,p})^k t^z,
\end{equation*}
the result of applying the right-hand side of \eqref{eqn:fileprod}
to $t^{z+n}$ is 
\begin{equation}\label{eqn:pfRHS}
 \sum_{k=0}^n f_{n-k} (a,b;q,p;B)([z]_{a,b;q,p})^k \,t^z.
\end{equation}
On the other hand, let us apply $\alpha_B$ to $t^{z+n}$. Consider
the $j$-th $x$ from the right and let $l_j$ be the number of $y$'s
to the right of the $j$-th $x$. Applying the $j$-th $x$ to $t^{z+n}$
gives $[z+n-l_j]_{a,b;q,p}$ and since there are $(n-l_j)$ many $y$'s
to the left of the $j$-th $x$, this factor becomes
$[z+n -l_j]_{aq^{-2(n-l_j)}, bq^{-(n-l_j)};q,p}$ in the end. There are
$n$ $x$'s and each $x$ gives such a factor. Thus applying $\alpha_B$ to 
$t^{z+n}$ gives 
\begin{equation*}
\prod_{j=1}^n [z+n -l_j]_{aq^{-2(n-l_j)}, bq^{-(n-l_j)};q,p}t^z.
\end{equation*}
Notice that $n-l_{n-j+1}$ equals the $j$-th column height $b_j$ of $B$. Hence,
\begin{equation*}\prod_{j=1}^n [z+n -l_j]_{aq^{-2(n-l_j)},
bq^{-(n-l_j)};q,p}t^z=\prod_{j=1}^n [z+b_j]_{aq^{-2b_j}, bq^{-b_j};q,p}t^z.
\end{equation*}
Comparing the coefficients of $t^z$ on both sides, we have proved the theorem.
\end{proof}

\begin{remark}
There is a vast amount of literature on normal ordering.
We only hint at some of the references which are most relevant
for the present paper.
In \cite{BF}, different combinatorial models (involving special graphs
and diagrams) are surveyed and studied to reduce
elements of the Weyl algebra to normally ordered form.
In \cite{MS0}, $q$-weights on Feynman diagrams are used to 
derive normally ordered forms in the $q$-Weyl algebra.
The algebra generated by $x$ and $y$ satisfying $xy=yx+h y^s$,
$h\in\mathbb C$ and $s\in\N_0$,
has been considered in \cite{Var} where the author showed that
the normal order coefficients are the $i$-rook numbers of the $i$-creation
rook model defined by Goldman and Haglund \cite{GH}.
Further cases have been considered in the $q$-commuting case.
In \cite{R} a binomial formula was established for variables
$x$ and $y$ satisfying the quadratic commutation relation
$xy=ax^2+ q yx+h y^2$.
The paper \cite{MS} thoroughly studies the commutation relation
$xy=yx+hf(y)$ for various choices of the functions $f$.
A detailed survey is provided in \cite{MS1};
for more material, see the references given in \cite{BF,MS,MS1}.
It is natural to ask for elliptic extensions of all these results.
\end{remark}


\section{Final remarks}


\begin{enumerate}

\item
The $a;q$-weight can be expressed as 
\begin{displaymath}
w_{a,0;q}(s,t)=\frac{(1-a q^{s+2t})}{(1-a q^{s+2t-2})}q^{-1}=
\frac{q^{-\frac s2-t}/\sqrt{a} -
\sqrt{a} q^{\frac s2+t}}{q^{-\frac s2-t+1}/\sqrt{a}-
\sqrt{a} q^{\frac s2+t-1}}.
\end{displaymath}
(The choice of the square root of $a$ does not matter, as long as it is the
same everywhere.)
If we let $q=e^{ix}$, $\sqrt{a}=e^{i(\alpha+1)x}$, then 
\begin{align*}
w_{a,0;q}(s,t)&=\frac{e^{-i(\alpha+\frac s2+t+1)x}-
e^{i(\alpha+\frac s2+t+1)x}}
{e^{-i(\alpha+\frac s2+t)x}-e^{i(\alpha+\frac s2+t)x}}=
\frac{\sin (\alpha+\frac s2+t+1)x}{\sin (\alpha+\frac s2+t)x}\\
&= \frac{U_{\alpha+\frac s2+t}(x)}{U_{\alpha +\frac s2+t-1}(x)}, 
\end{align*}
where $U_n (\cos \theta)=\frac{\sin (n+1)\theta}{\sin \theta}$ is the
\emph{Chebyshev polynomial of the second kind}. Hence the $a;q$-weight
$w_{a,0;q}(s,t)$ can be considered as a quotient of generalized
Chebyshev polynomials. Thus the identities related to the
$a;q$-weights can be reformulated as identities for Chebyshev polynomials
of the second kind.

\item
The $a;q$-binomial coefficients \eqref{eqn:aqbincoeff} are symmetric in
$(k,n-k)$ (whereas the more general elliptic binomial coefficients
and also the $b;q$-binomial coefficients \eqref{eqn:bqbincoeff}
are not symmetric).
They satisfy
the two recurrence relations
\begin{align*}
\begin{bmatrix}n+1\\k\end{bmatrix}_{a;q}&=
\begin{bmatrix}n\\k\end{bmatrix}_{a;q}+
\frac{(1-aq^{2n+2-k})}{(1-aq^k)}q^{k-n-1}
\begin{bmatrix}n\\k-1\end{bmatrix}_{a;q},\\
\begin{bmatrix}n+1\\k\end{bmatrix}_{a;q}&=
\frac{(1-aq^{n+1+k})}{(1-aq^{n+1-k})}q^{-k}
\begin{bmatrix}n\\k\end{bmatrix}_{a;q}+
\begin{bmatrix}n\\k-1\end{bmatrix}_{a;q},
\end{align*}
which, together with the initial conditions
\begin{equation*}
\begin{bmatrix}0\\0\end{bmatrix}_{a;q}=1,\qquad\text{and}\quad
\begin{bmatrix}n\\k\end{bmatrix}_{a;q}=0,\qquad\text{for
$k>n$ or $k<0$}, \end{equation*}
determine them uniquely.
This symmetry might be reason that combinatorial
enumeration using $a;q$-weights appears to lead more often
to closed forms than combinatorial enumeration using elliptic or $b;q$-weights.
See \cite{SY1} for some
basic hypergeometric series identities proved using
$a;q$-weights.

\item
It is important to realize that every identity involving
the variables $a$, $b$, $x$ and $y$
respecting the commutative relations \eqref{abqxy}
also holds when they are replaced by the variables
$b$, $a$, $y$ and $x$, respectively, due to the symmetry of
\eqref{abqxy}.

In this way we can immediately deduce various additional results such as
$a;q$-versions from corresponding $b;q$-versions (and vice-versa), etc.
To single out a particular result, we note that
the following functional equation of $a;q$-exponentials
\begin{equation}
e_{a;q}(x+y)=e_{a;q}(y)e_{a;q}(x)
\end{equation}
holds where, as in \eqref{bqexp},
\begin{equation}\label{aqexp}
e_{a;q}(z):=\sum_{n=0}^{\infty}\frac{1}{(q;q)_n (aq;q)_n}z^n.
\end{equation}

\item
We were able to establish functional equations for $a;q$- and for
$b;q$-exponentials. As a matter of fact, we were not able to
unify both results and obtain a nice functional equation
for $a,b;q$-exponentials or $a,b;q,p$-exponentials.

\end{enumerate}


\end{document}